\documentclass{article}
\usepackage{amsmath, amsthm, amssymb}
\usepackage{enumerate,mathrsfs,setspace}
\usepackage{url}
\usepackage{tikz}
\usepackage{verbatim}

\title{Mixing Homomorphisms, Recolourings, and Extending Circular Precolourings}
\author{Richard~C.~Brewster\thanks{Research supported by the Natural Science and Engineering Research Council of Canada (NSERC).} \\
Dept of Math and Stats \\
Thompson Rivers University \\
{\small \url{rbrewster@tru.ca}}
\and 
Jonathan~A.~Noel\thanks{Research done at Thompson Rivers University and
supported by the NSERC Undergraduate Student Research Awards program.} \\
Mathematical Institute \\ University of Oxford \\
{\small \url{noel@maths.ox.ac.uk}}
}

\newtheorem{thm}{Theorem}[section]
\newtheorem{lem}[thm]{Lemma}
\newtheorem{prop}[thm]{Proposition}
\newtheorem{conj}[thm]{Conjecture}
\newtheorem{ques}[thm]{Question}
\newtheorem{claim}[thm]{Claim}
\newtheorem{cor}[thm]{Corollary}

\newtheorem*{thm:ext}{Theorem~\ref{thm:extcirc}}

\theoremstyle{definition}
\newtheorem{defn}[thm]{Definition}

\newtheorem{case}{Case}

\newtheorem{note}[thm]{Note}

\theoremstyle{remark}
\newtheorem{rem}[thm]{Remark}

\begin{document}
\maketitle

\begin{abstract}
This work brings together ideas of mixing graph colourings, discrete homotopy, 
and precolouring extension. A particular focus is circular colourings. 
We prove that all the $(k,q)$-colourings of a graph $G$ can be obtained
by successively recolouring a single vertex provided $k/q\geq 2col(G)$ 
along the lines of Cereceda, van den Heuvel and Johnson's result
for $k$-colourings.  We give various bounds for such mixing results and
discuss their sharpness, including cases where the bounds for circular and classical
colourings coincide.
As a corollary, we obtain an Albertson-type extension theorem for
$(k,q)$-precolourings of circular cliques. Such a result was first conjectured
by Albertson and West. 

General results on homomorphism mixing are presented, including a characterization 
of graphs $G$ for which the endomorphism monoid can be generated through the mixing process. 
As in similar work of Brightwell and Winkler, the concept of dismantlability plays a key role.

\end{abstract}


%
%
%
%
%
%
%

%

\section{Introduction}

Given a graph $G$ and a $k$-colouring $c:V(G)\to\{0,1,\dots,k-1\}$ the \emph{mixing process} is the following:
\begin{enumerate}
\item Choose a vertex $v\in V(G)$;
\item Change the colour of $v$ (if possible) to yield a different $k$-colouring $c':V(G)\to\{0,1,\dots,k-1\}$.
\end{enumerate}
A natural problem arises: Can every $k$-colouring of $G$ eventually be generated by
repeating this process? If so, we say that $G$ is \emph{$k$-mixing}.
(As a notation convention throughout the paper, the term $k$-colouring
refers to \emph{proper $k$-colouring}.)

The problem of determining whether a graph is $k$-mixing and related
problems have been studied in a recent series of 
papers~\cite{PSPACE, Complexity,Connectedness,Btwn3colourings} (see also the survey of van den Heuvel~\cite{mixingsurvey}). 
The main objective of this paper is to investigate mixing  problems in 
the more general setting of graph homomorphisms, and specifically circular colourings.
We establish bounds on the number of colours required to ensure a graph $G$
is mixing with respect to circular colourings.  The bounds are in terms of the 
colouring number and clique number, along the lines of~\cite{Connectedness}. 
As a corollary, we obtain an Albertson-type extension theorem for
$(k,q)$-precolourings of circular cliques. Such a result was first conjectured
by Albertson and West~\cite{AlbWest} and studied in~\cite{BrewsterNoel}.
Using general results on homomorphism mixing, 
we characterize graphs for which the mixing process generates the entire 
endomorphism monoid (cf.~\cite{Winkler}). 

For general graph theory terminology and notation, we follow~\cite{BondyMurty}.
We consider finite undirected graphs without multiple edges. 
We are mainly interested in loop-free graphs, with the exception of 
Section~\ref{MixHom} where vertices with loops are a natural construct. 
A vertex with a loop is said to be \emph{reflexive}. 
The subgraph induced by the reflexive vertices is 
the \emph{reflexive subgraph}.   We use both $uv$ and $u \sim v$ to indicate
adjacency using the latter when needed for notational clarity. 

Given graphs $G$ and $H$, a \emph{homomorphism} of $G$ to $H$ is a mapping 
$\varphi:V(G)\to V(H)$ such that $\varphi(u)\varphi(v)\in E(H)$ whenever 
$uv\in E(G)$. We write $G\to H$ to indicate the existence of a homomorphism 
of $G$ to $H$, and write $\varphi:G\to H$ when referring to a specific
homomorphism $\varphi$. We let $HOM(G,H)$ denote the collection of 
homomorphisms $G\to H$. Within this framework, we can view a $k$-colouring of 
$G$ as a homomorphism $G\to K_k$. Consequently, we refer to the images of 
$\varphi$, i.e. vertices of $H$, as \emph{colours}, and say $G$ is 
\emph{$H$-colourable}.
We refer the reader to~\cite{HellNesetril} for an in-depth study of graph
homomorphisms.

Given positive integers $k,q$ with $k \geq 2q$, the \emph{circular clique},
$G_{k,q}$ has vertex set $\{ 0, 1, \dots, k-1 \}$ with $ij$ an edge when 
$q \leq |i-j| \leq k-q$.
A homomorphism $\varphi:G\to G_{k,q}$ is called a \emph{circular colouring}
in general, and a \emph{$(k,q)$-colouring} of $G$ for the specific pair $(k,q)$. 
We remark that $G_{k,1}$ is isomorphic to $K_k$ and so a $(k,1)$-colouring
is simply a $k$-colouring. As it turns out, $G_{k',q'}\to G_{k,q}$ if and only if 
$k'/q'\leq k/q$ (see for example~\cite{BondyHell}). 
The \emph{circular chromatic number} of a graph $G$ is defined 
analogously to the chromatic number:
\[\chi_c(G):=\inf\{k/q:G\to G_{k,q}\}.\] 

It is well known~\cite{Vince} that $\chi(G)=\left\lceil\chi_c(G)\right\rceil$
and $\chi_c(G)=\min\{k/q:G\to G_{k,q}\}$. 
Bondy and Hell gave a purely combinatorial proof of the latter result
in~\cite{BondyHell}. In their proof, they use the fact that 
non-(vertex)-surjective $(k,q)$-colourings cannot minimize $k/q$. 
That is, if $\gcd(k,q)=1$, then there is a pair $(k',q')$ such that $k'/q'<k/q$ 
and $(G_{k,q}-0)\to G_{k',q'}$. In this paper, we often exploit this mapping. 
The relationship between $(k',q')$ and $(k,q)$ is captured by the following.

\begin{defn}\label{defn:lowerparent}
Let $k,q,k',q'$ be non-negative integers where $q,q'\neq0$ and 
$\gcd(k,q)=1$. Further, suppose that
\[kq'-k'q=1.\]
If $q\geq q'\geq1$, then $(k',q')$ is said to be the \emph{lower parent} 
of $(k,q)$. 
\end{defn}

It can be shown that lower parents always exist and are unique. See, for 
example, Lemma~6.6 of~\cite{HellNesetril}. Lower parents are related to Farey sequences, where the lower parent $k'/q'$ 
immediately precedes $k/q$ in the Farey sequence $F_q$.
For more background on the circular chromatic number, we refer the reader 
to the survey of Zhu~\cite{Zhu}.

As is done in~\cite{Connectedness}, we view mixing problems in terms of the 
connectedness of a certain auxiliary graph. Given graphs 
$G$ and $H$ such that $G\to H$, define the \emph{$H$-colour graph} of $G$, denoted 
$\mathscr{C}_H(G)$, to be the graph on vertex set $HOM(G,H)$ where $f\sim g$ 
if $f(v)\neq g(v)$ for exactly one $v\in V(G)$. The graph $G$ is said to be 
\emph{$H$-mixing} if $\mathscr{C}_H(G)$ is connected.  

For colourings, let us define the \emph{$k$-colour graph} and \emph{$(k,q)$-colour graph} 
of $G$ by $\mathscr{C}_k(G):=\mathscr{C}_{K_k}(G)$ and 
$\mathscr{C}_{k,q}(G):=\mathscr{C}_{G_{k,q}}(G)$ respectively. Thus, $G$ is 
$k$-mixing or $(k,q)$-mixing if the corresponding colour graph is connected.

In~\cite{Connectedness} a family $L_m, m\geq 4$, of bipartite graphs is constructed
with the property that $L_m$ is $k$-mixing if and only if $k \geq 3$ and $k \neq m$.  
Consider the following definition which captures this situation.

\begin{defn}\label{defn:mt}
Define the \emph{mixing number} $\mathfrak{m}(G)$ and \emph{circular mixing
number} $\mathfrak{m}_c(G)$ by:
\begin{eqnarray*}
\mathfrak{m}(G) & := & \min\{k\in\mathbb{N}: k\geq\chi(G) \text{ and } G \text{ is } k\text{-mixing}\}; \\
\mathfrak{m}_c(G) & := & \inf\{k/q\in\mathbb{Q}: k/q\geq\chi_c(G) \text{ and } G \text{ is } (k,q)\text{-mixing}\}.
\end{eqnarray*}
Define the \emph{mixing threshold} $\mathfrak{M}(G)$ and \emph{circular mixing threshold} $\mathfrak{M}_c(G)$ as:
\begin{eqnarray*}
\mathfrak{M}(G) & := & \min\{r\in\mathbb{N}: r\geq\chi(G) \text{ and } G \text{ is } k\text{-mixing for all } k\geq r\}; \\
\mathfrak{M}_c(G) & := &\inf\{r\in\mathbb{Q}: r\geq\chi_c(G) \text{ and } G \text{ is } (k,q)\text{-mixing for all } k/q\geq r\}.
\end{eqnarray*}
\end{defn}

For the family $L_m$ mentioned above we have $\chi(L_m) = 2, \mathfrak{m}(L_m) = 3$, and $\mathfrak{M}(L_m) = m+1$. 
Thus $\mathfrak{M}$ is not bounded above by any function of $\mathfrak{m}$ (or of $\chi$).
%

\subsection*{Our Results}

In Section~\ref{degen} we show $\mathfrak{M}_c(G) \leq 2 \, col(G)$, where
$col(G)$ is the \emph{colouring number} of $G$ (defined below).
This is analogous to the result $\mathfrak{M}(G) \leq col(G)+1$ proved in~\cite{Connectedness}.  
The factor of $2$ required for the circular colouring bound versus the (classical) colouring bound 
is consistent with similar bounds on the list chromatic number for circular and classical colourings
respectively.

In Section~\ref{MixHom} we study mixing in general 
homomorphisms, where we identify a connection between mixing problems and a 
homotopy theory for graph homomorphisms. Specifically for $G$ and $H$,
we generalize the colouring graph of~\cite{Connectedness} to homomorphisms 
and show the connected components correspond to the connected components
of the reflexive subgraph of $H^G$ (see Proposition~\ref{colourHom}).
We obtain a characterization of graphs $G$ which are $G$-mixing 
(i.e. $G$ for which the endomorphism monoid can be generated by the mixing process) using the
concept of dismantability.  We also use dimantability to show for trees and complete
bipartite graphs $\mathfrak{M}_c = \mathfrak{m}_c = 2$ (see Proposition~\ref{K_mmtrees}).

Section~\ref{parents} is devoted to investigating the role of non-surjective colourings and 
lower parents in $(k,q)$-mixing problems with one goal being the establishment of upper
bounds on $\mathfrak{M}_c$.  In~\cite{BrewsterNoel} we give examples
where $\mathfrak{M}_c > \mathfrak{M}$.  Here we give the following upper bounds:
$\mathfrak{M}_c \leq 2 \Delta(G)$ (provided $G$ has at least one edge), see Theorem~\ref{2Delta}.
In addition, we show in Theorem~\ref{v+1/2}
\[\mathfrak{M}_c(G)\leq\max\left\{\frac{|V(G)|+1}{2},\mathfrak{M}(G)\right\}.\]

In Section~\ref{lb} we establish the lower bound $\mathfrak{m}_c \geq \max \{ 4, \omega + 1 \}$ 
for non-bipartite graphs where $\omega$ denotes the clique number. Extension problems for homomorphisms are studied in Section~\ref{ext}.
We show that Albertson-type extension theorems can often be deduced from results on mixing. 
As a result, we obtain a better understanding of the problem of extending precolourings of circular cliques, 
which was first raised by a conjecture of Albertson and West~\cite{AlbWest}. Specifically
we show
\begin{thm:ext}
For $k\geq2q$ and $\gcd(k,q)=1$, let $X$ be a $(k,q)$-colourable graph containing disjoint copies $X_1,X_2,\dots,X_t$ 
each isomorphic to $G_{k,q}$ and suppose that $\frac{k'}{q'}\geq \max\left\{\frac{k+1}{2}, \left\lceil\frac{k}{q}\right\rceil+1\right\}$. 
Then there exists a distance $d$ such that if $f_i:X_i\to G_{k',q'}$ is a homomorphism for $i=1,2,\dots,t$ and $d_{X}(X_i,X_j)\geq d$ 
for $i\neq j$, then the precolouring $f_1\cup\cdots\cup f_t$ extends to a $(k',q')$-colouring of $X$.
\end{thm:ext}

The final section contains examples to illustrate the sharpness of our bounds, 
and indicates several questions for future study.

%
%
\section{The colouring number}
\label{degen}

The \emph{colouring number} of a graph $G$ is $col(G):=\max\{\delta(H)+1:H\subseteq G\}$. 
The vertices of $G$ can be ordered $v_1, \dots, v_n$ so
that $v_i$ has at most $col(G)-1$ neighbours in $v_1, \dots, v_{i-1}$. 
A greedy colouring of $G$ using this ordering requires at most $col(G)$ colours.
Recall, a $k$-list assignment $L$ for a graph $G$ is
a function which assigns to each vertex of $G$ a set of at least $k$ colours.
An $L$-list colouring $f$ is a proper colouring of $G$ such that $f(v)\in L(v)$ 
for all $v\in V(G)$. The \emph{list chromatic number} $\chi_\ell(G)$ is the minimum $k$ 
such that $G$ has an $L$-list colouring for every $k$-list assignment $L$. 
Given any $col(G)$-list assignment of $G$, if we use the ordering above in a greedy colouring, 
then when we come to colour $v_i$ it has at most $col(G)-1$ 
coloured neighbours. Thus, we are guaranteed to have at least one available colour to use on $v_i$. 
This implies the following well-known bound.
\begin{equation}\label{list}
\chi(G) \leq \chi_\ell(G)\leq col(G).
\end{equation}

In mixing problems, one may require two available 
colours at $v_i$ to facilitate a local modification. This fact was used 
in~\cite{Connectedness} to prove the following tight bound.
\begin{thm}[Cereceda et al.~\cite{Connectedness}]
\label{col}
For any graph $G$, $\mathfrak{M}(G)\leq col(G)+1$.
\end{thm}
In fact, Choo and MacGillivrary~\cite{Gray} proved that if $k\geq col(G)+2$, then $\mathscr{C}_k(G)$ is Hamiltonian; i.e. there is a cyclic Gray code for the 
$k$-colourings of $G$.

A circular analog of the list chromatic number was introduced by Mohar~\cite{Mohar},
and subsequently studied by many authors~\cite{Havet,Tobias,TwoQs,CircList}. The 
\emph{circular list chromatic number} $\chi_{c,\ell}$ of $G$
is the infimum over $t\in\mathbb{Q}$ such that for every $\lceil tq\rceil$-list 
assignment $L$ from $\{0,1,\dots,k-1\}$ there is a $(k,q)$-colouring $f$ of $G$
with $f(v)\in L(v)$ for all $v\in V(G)$. 
In~\cite{CircList} Zhu
obtains the following bound in terms of the colouring number
\begin{equation}\label{clist}\chi_{c,\ell}(G)\leq 2(col(G)-1).\end{equation}

Our main result in this section is the generalization of Theorem~\ref{col} to the
circular setting. As with Zhu's result above, the bound for mixing circular colourings requires a factor of $2$
compared with classical colouring.  Our result is the following:

\begin{thm}
\label{col(G)Bound}
For a graph $G$, $\mathfrak{M}_c(G)\leq 2col(G)$. 
\end{thm}

The proof of Theorem~\ref{col(G)Bound} is based on the mixing process for 
homomorphisms; namely, given $f:G \to H$:
\begin{enumerate}
\item Choose a vertex $v\in V(G)$;
\item Change $f(v)$ (if possible) to yield a homomorphism $f':G \to H$.
\end{enumerate}
For classical colourings, one only requires that $f'(v)$ 
differs from the colours assigned to $N(v)$. 
For homomorphisms, $f'(v)$ must be adjacent (in $H$) with all the colours 
assigned to vertices in $N(v)$.   
That is, $f'(v)$ must belong to the following set:
\begin{equation}\label{eqn:available}
\bigcap_{u \in N_G(v)} N_H(f(u)).
\end{equation}

In the case of $(k,q)$-colourings it is convenient to 
speak in terms of intervals of colours. Given $i,j\in\{0,1,\dots,k-1\}$
let $[i,j]$ denote the set $\{i,i+1,\dots, j\}$, where colours
are reduced modulo $k$.  Let $f: G \to G_{k,q}$.
Since each colour $c$ in $G_{k,q}$ has an interval of 
non-neighbours $[c-q+1,c+q-1]$ of size $2q-1$, we obtain the expression
for the set of available colours to recolour a vertex $v$:
\begin{equation}\label{eqn:availcirc}
\bigcap_{u \in N_G(v)} N_{G_{k,q}}(f(u)) = 
V(G_{k,q}) - \bigcup_{u \in N_G(v)} [ f(u) - q + 1, f(u) + q - 1].
\end{equation}
Hence the number of available 
colours for $v$ is at least $k-(2q-1)d(v)$. 
If this number is in turn at least $2q$, we can apply an inductive argument 
along the lines of~\cite[Theorem 3]{Connectedness}.

\begin{lem}
\label{degree}
Let $G$ be a graph with $v\in V(G)$. Suppose integers $k,q$ satisfy $k\geq(2q-1)d(v)+2q$. 
If $G-v$ is $(k,q)$-mixing, then $G$ is $(k,q)$-mixing.
\end{lem}

\begin{proof}
Denote $G-v$ by $G'$. Let $f$ and $g$ be arbitrary 
$(k,q)$-colourings of $G$ and let $f'$ and $g'$ denote
the restrictions of $f$ and $g$ to $G'$ respectively. 
We show that there is a path from $f$ to $g$ in 
$\mathscr{C}_{k,q}(G)$. The proof is by induction
on the distance $d(f',g')$ in $\mathscr{C}_{k,q}(G')$. 
In the base case $d(f',g')=0$, the colourings
$f$ and $g$ differ on at most one vertex, 
namely $v$, and so either $f=g$ or $f\sim g$.

Now, suppose $d(f',g')=d\geq1$. Let $h'$
be a $(k,q)$-colouring of $G'$ adjacent to 
$f'$ such that $d(h',g')=d-1$. Using $h'$ we 
construct a colouring $h$ of $G$ admitting
paths to both $f$ and $g$. 

Define $S:=f'(N_G(v))\cup h'(N_G(v))$. 
Since $|N(v)|=d(v)$ and $f'$ and $h'$ differ on only one 
vertex, we have $|S|\leq d(v)+1$. 
By hypotheses $k\geq(2q-1)d(v)+2q$, thus, $k\geq (2q-1)|S|+1$.
Consequently, there is a colour $c$ adjacent (in $G_{k,q}$) to every 
colour in $S$. Define, a proper colouring, $h:G\to G_{k,q}$ by:
\begin{displaymath}h(u)=\left\{\begin{array}{ll}  
                 c  & \mbox{ if }  u=v,  \\  
                 h'(u) &  \mbox{ otherwise. }    
               \end{array}   \right.      \end{displaymath}
Clearly the restriction of $h$ to $G'$ is $h'$. 
Since $d(h',g')=d-1$, there exists a path from
$h$ to $g$ in $\mathscr{C}_{k,q}(G)$ by the
inductive hypothesis.

Finally, we show that there is a path from $f$
to $h$ in $\mathscr{C}_{k,q}(G)$. Let $w$ be the
unique vertex such that $f'(w)\neq h'(w)$. Define
a colouring $j:G\to G_{k,q}$ in the following way:
\begin{displaymath}j(u)=\left\{\begin{array}{ll}  
                 c  & \mbox{ if }  u=v,  \\  
                 f'(u) &  \mbox{ otherwise. }    
               \end{array}   \right.      \end{displaymath}
Again, the fact that $j$ is a $(k,q)$-colouring follows
from our definition of $c$. The colourings $f$ and $j$ 
differ on at most one vertex, namely $v$. Also, $j$
only differs from $h$ on $w$. Therefore, there is
a path from $f$ to $h$ (through $j$). The result 
follows.
\end{proof}

\begin{proof}[Proof of Theorem~\ref{col(G)Bound}]
Suppose $k/q \geq 2 col(G)$. 
We show $G$ is $(k,q)$-mixing by induction on $|V(G)|$, where the base case
$|V(G)|=1$ is considered trivial. Suppose that $|V(G)|\geq2$.
Clearly $G$ has a vertex $v$ such that $d(v)\leq col(G)-1$.
The graph $G'=G-v$ satisfies $k/q \geq 2col(G')$ since $col(G) \geq col(G')$.
Thus $G'$ is $(k,q)$-mixing by the inductive hypothesis.
Also $k \geq 2q col(G) \geq 2q (d(v)+1) = 2qd(v) + 2q \geq (2q-1)d(v) + 2q$.
This implies that $G$ is $(k,q)$-mixing by Lemma~\ref{degree}.
\end{proof}

%
%
\section{Homotopy}
\label{MixHom}

\begin{note}
In this section, we examine the well studied \emph{homomorphism graph} which is related
to the colour graph defined above.  By definition, the homomorphism graph has a loop
on each vertex. Thus, in this section we allow loops unless otherwise stated.
\end{note}

Given graphs $G$ and $H$, we recall the \emph{categorical product} $G \times H$ is
the graph defined by:
\begin{itemize}
  \item $V(G \times H) = V(G) \times V(H)$;
  \item $(g_1,h_1)(g_2,h_2) \in E(G \times H)$ if and only if 
  $g_1 g_2 \in E(G)$ and $h_1 h_2 \in E(H)$.
\end{itemize}

Following~\cite{Dochtermann},
we recall, for topological spaces $X$ and $Y$,
maps $f: X \to Y$ and $g: X \to Y$ are homotopic if there is an appropriately defined
map $\varphi: [0, 1] \to HOM(X,Y)$ such that $\varphi(0) = f$ and $\varphi(1) = g$. 
Recently an analogous theory of homotopy has been developed for graph 
homomorphisms~\cite{Dochtermann, Dochtermann2}. 
Given graphs $G$ and $H$, define the 
\emph{$H$-homomorphism graph} of $G$, denoted $\mathscr{H}_H(G)$, to be 
the graph on vertex set $HOM(G,H)$ where $f\sim g$ if $f(u)g(v)\in E(H)$ 
for all $uv\in E(G)$. By definition, $\mathscr{H}_H(G)$ is reflexive.
We say that $f:G\to H$ and $g:G\to H$ are \emph{$\times$-homotopic} if there is
a path from $f$ to $g$ in $\mathscr{H}_H(G)$.  As noted in~\cite{Dochtermann} 
such a path corresponds to a homomorphism $\varphi: I_n \to \mathscr{H}_H(G)$ 
where $I_n$ is the reflexive path on $n$ vertices.  Equivalently, such a homotopy is a 
homomorphism $\tilde{\varphi} : G \times I_n \to H$ where $\tilde{\varphi}(-,0) = f$ and 
$\tilde{\varphi}(-,n-1) = g$.  This connection between $\mathscr{H}_H(G)$ and 
the product $G \times I_n$ indicates why we use the notation $\times$-homotopy.
%
%

A connection between the $\times$-homotopy classes of graph homomorphisms and the topology 
of Hom complexes is established by Dochtermann in~\cite{Dochtermann}. Lov{\'a}sz first introduced Hom complexes
in his proof of Kneser's Conjecture~\cite{LovaszK}. 
The $H$-homomorphism graph is perhaps better known as the reflexive
subgraph of the \emph{exponential graph} $H^G$. 
Exponential graphs were introduced by Lov{\'a}sz in~\cite{Lovasz} and have 
been used to study and solve many interesting homomorphism and colouring 
problems, see for example~\cite{BrewsterGraves,El-Z}. 
For loop-free graphs, the connected components of $\mathscr{C}_H(G)$ correspond
to \emph{mixing classes}, whereas the connected components of $\mathscr{H}_H(G)$ correspond
to \emph{homotopy classes}.  The following proposition shows that, for loop-free graphs, these classes coincide.

\begin{prop}
\label{colourHom}
Let $G$ and $H$ be graphs such that $G$ is loop-free and $G \to H$.
Then $\mathscr{C}_H(G)$ is a spanning subgraph of $\mathscr{H}_H(G)$. 
Moreover, there is a walk from $f$ to $g$ in $\mathscr{H}_H(G)$ 
if and only if there is a walk from $f$ to $g$ in $\mathscr{C}_H(G)$.
\end{prop}

\begin{proof}
Clearly the $H$-colour graph and the $H$-homomorphism graph of $G$ have the 
same vertex set, namely $HOM(G,H)$. To see that $\mathscr{C}_H(G)$ is a 
subgraph of $\mathscr{H}_H(G)$ let $f$ and $g$ be adjacent in 
$\mathscr{C}_H(G)$. Suppose $uv\in E(G)$. Since $G$ is loop-free and $f$ and $g$ differ on one 
vertex, we may assume $f(u)=g(u)$. It follows that $f(u)g(v)\in E(G)$ since 
$f$ and $g$ are homomorphisms. Thus, $fg\in E(\mathscr{H}_H(G))$. Hence,
$\mathscr{C}_H(G) \subseteq \mathscr{H}_H(G)$, and every walk in
$\mathscr{C}_H(G)$ is a walk in $\mathscr{H}_H(G)$.

Thus suppose there is a walk in $\mathscr{H}_H(G)$ from $f$ to $g$.  In fact
we consider the case where $fg$ is an edge, from which the general situation
follows. 
We construct an $(f,g)$-walk in $\mathscr{C}_H(G)$. If $f=g$ there is nothing 
to prove, thus, suppose that $f$ and $g$ differ on $t\geq1$ vertices, say 
$v_1,v_2,\dots ,v_t$. For $0\leq j\leq t$ we define the mapping 
$h_j:V(G)\to V(H)$ in the following way:

\begin{displaymath}
h_j(v)=\left\{\begin{array}{ll}  
                 g(v)  & \mbox{ if }  v=v_i \mbox{ for } i\leq j,  \\  
                 f(v) &  \mbox{ otherwise. }    
               \end{array}   \right.
\end{displaymath}

Clearly we have that $h_0=f$, $h_t=g$. We also observe that for 
$0\leq j\leq t-1$ the map $h_j$ and $h_{j+1}$ differ on exactly one vertex. 
To show that $h_0h_1\dots h_t$ is our desired $(f,g)$-walk, we must show each 
$h_j$ is a homomorphism. This is easily done. Observe for all $v \in V(G)$, $h_j(v)$ is equal to 
$f(v)$ or $g(v)$. For $uv\in E(G)$ we have 
$f(u)f(v), f(u)g(v), g(u)g(v)\in E(H)$ since $f$ and $g$ are homomorphisms 
and $fg\in E(\mathscr{H}_H(G))$. The result follows. 
\end{proof}

Given the equivalence between mixing classes and homotopy classes, 
when studying mixing problems for loop-free graphs we may use either
$\mathscr{C}_H(G)$ or $\mathscr{H}_H(G)$,  whichever 
is most convenient for the problem at hand.
Observe the above proposition does not hold for graphs with loops allowed. For example, if $G$ is the
graph consisting of two isolated reflexive vertices $u$ and $v$, then there are precisely
four functions mapping $V(G)\to V(G)$, all of which are homomorphisms $G\to G$. 
No two homomorphisms $G\to G$ are $\times$-homotopic, yet $\mathscr{C}_G(G)$ is connected, see 
Figure~\ref{twoisolated}. 

\begin{figure}
\begin{center}
\begin{tikzpicture}
  [std/.style={circle, draw=black!100,fill=black!100,thick, inner sep=1pt, minimum size=1.5mm},
   filled/.style={circle, draw=black!100,fill=black!20,thick, inner sep=1pt, minimum size=3mm}]

   \node[std]         (v0) 		at (0,0)         	{};
   \node[std] 		(v1)		at (0,1.5)		{}; 

   \node[left] at (v0.west) {$u$};
   \node[left] at (v1.west) {$v$};

   \draw (v0) to [out=60,in=0] (0,0.5) to [out=180,in=120] (v0);  
   \draw (v1) to [out=60,in=0] (0,2) to [out=180,in=120] (v1);  
   
   \node at (-0.7,1) {$G$};

%
%
   \node[filled]         (c00) 		at (3,0)         	{$uu$};
   \node[filled] 		(c01)		at (3,1.5)	{$uv$}; 
   \node[filled]         (c10) 		at (4.5,0)      	{$vu$};
   \node[filled] 		(c11)		at (4.5,1.5)	{$vv$}; 

   \draw (c00) -- (c01) -- (c11) -- (c10) -- (c00);  
   
   \node at (2.2,0.75) {$\mathscr{C}_G(G)$};

   \node[filled]         (h00) 		at (8,0)         	{$uu$};
   \node[filled] 		(h01)		at (8,1.5)	{$uv$}; 
   \node[filled]         (h10) 		at (9.5,0)      	{$vu$};
   \node[filled] 		(h11)		at (9.5,1.5)	{$vv$}; 

   \draw (h00) to [out=60,in=0] (8,0.7) to [out=180,in=120] (h00);  
   \draw (h01) to [out=60,in=0] (8,2.2) to [out=180,in=120] (h01);  
   \draw (h10) to [out=60,in=0] (9.5,0.7) to [out=180,in=120] (h10);  
   \draw (h11) to [out=60,in=0] (9.5,2.2) to [out=180,in=120] (h11);  

   \node at (7.2,0.75) {$\mathscr{H}_G(G)$};

\end{tikzpicture}
\end{center}
\caption{A reflexive graph $G$ such that $\mathscr{C}_G(G)\nsubseteq\mathscr{H}_G(G)$.  The mapping $f : G \to G$ is encoded
as a two digit string $f(u)f(v)$.}
\label{twoisolated}
\end{figure}
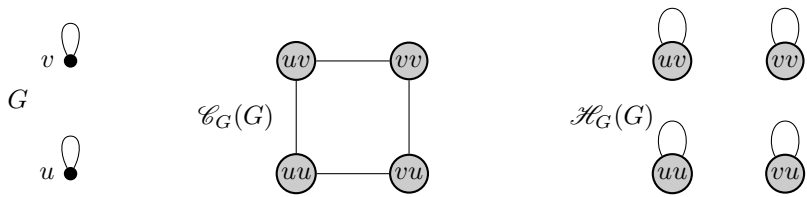

\begin{defn}
\label{specialHoms}
Let $X$ and $Y$ be graphs and let $\varphi:X\to Y$ be a homomorphism. Then $\varphi$ is 
\begin{enumerate}[(a)]
\item an \emph{endomorphism} if $Y=X$.
\item a \emph{retraction} if there is a homomorphism $\psi:Y\to X$, called a \emph{section}
(corresponding to $\varphi$), such that $\varphi\psi$ is the identity on $Y$. We say that $X$
\emph{retracts} to $Y$, or that $Y$ is a \emph{retract} of $X$.  
\item a \emph{fold} if $Y=X-v$ for some $v\in V(X)$, $\varphi$ is the identity on $Y$, and $N(v)\subseteq N(\varphi(v))$. We say that $X$ \emph{folds} to $Y$.
\item a \emph{dismantling retraction} if $\varphi$ is a (possibly empty) composition of folds. We say that $X$ \emph{dismantles} to $Y$.
\end{enumerate}
\end{defn}

%
\begin{defn}
\label{rigidcorestiff}
Let $X$ be a graph. Then $X$ is said to be 
\begin{enumerate}[(a)]
\item \emph{rigid} if the only endomorphism of $X$ is the identity.
\item a \emph{core} if $X$ does not retract to a proper subgraph.
\item \emph{stiff} if $X$ does not fold to a proper subgraph.
\item \emph{dismantlable} if $X$ has a dismantling retraction $\varphi:X\to Y$, where $Y$ is a rigid graph.
\end{enumerate}
\end{defn}

\begin{rem}
It can be seen that rigid $\Rightarrow$ core $\Rightarrow$ stiff, but the converse is false for both implications.
\end{rem}

We note that dismantlability is often only defined for 
reflexive graphs, and so our definition is not standard.   
However, our definition coincides with the
standard definition when applied to reflexive graphs.
That is, the only rigid graph with loops is the graph $\mathbf{1}$ consisting of a 
single reflexive vertex, and so a graph with loops is dismantlable 
if it has a dismantling retraction to $\mathbf{1}$. 
Dismantlability is often associated with the `cops and robbers' pursuit game on graphs. 
It was independently proven in~\cite{Nowakowski} and~\cite{Quilliot} that a reflexive graph
is cop-win (meaning that the cop has a winning strategy)
if and only if it is dismantlable. 

\begin{thm}\label{uniqueCoreStiff}
Let $X$ be a graph.
\begin{enumerate}[(a)]
\item \emph{(Hell and Ne{\v{s}}et{\v{r}}il~\cite{Core,HellNesetril})} There is a unique (up to isomorphism) subgraph $Y$ of $X$ such that $Y$ is a core and $X$ retracts to $Y$. The graph $Y$ is called the \emph{core} of $X$, denoted $Y=core(X)$.
\item \emph{(B{\'e}langer et al.~\cite{Belanger}; Fieux and Lacaze~\cite{Fieux}; Hell and Ne{\v{s}}et{\v{r}}il~\cite{HellNesetril})} There is a unique (up to isomorphism) subgraph $Y$ of $X$ such that $Y$ is stiff and $X$ has a dismantling retraction $\varphi:X\to Y$.\label{uniqueCore}
\end{enumerate}
\end{thm}

\begin{figure}
\begin{center}
\begin{tikzpicture}[node distance=2cm, auto]
  \node (G) {$G$};
  \node (H) [right of=G] {$H$};
  \node (F) [below of=G] {$F$};
  \node (K) [below of=H] {$K$};
  
  \draw[->] (G.250) to node [swap] {$g$} (F.110);
  \draw[->] (F.70) to node [swap]{$f$} (G.290);
  \draw[->] (H.250) to node [swap] {$h$} (K.110);
  \draw[->] (K.70) to node [swap]{$k$} (H.290);
  
  \draw[->] (G) to node {$x$} (H);
  \draw[->] (F) to node [swap] {$y$} (K);

\end{tikzpicture}
\end{center}
\caption{Diagram used to define $\varphi_h^f$ and $\varphi_k^g$, where $x \in HOM(G,H)$ and $y \in HOM(F,K)$.}
\label{fig:diagram}
\end{figure}
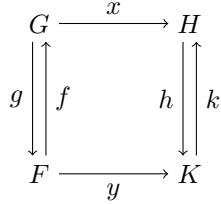


We will be investigating the relationship between $\mathscr{H}_H(G)$ and $\mathscr{H}_K(F)$ 
when there are homomorphisms $f,g,h,$ and $k$ as in
Figure~\ref{fig:diagram}.
Of special importance is the case where $g$ and $h$ are either both retractions or dismantling retractions. 
Define
\begin{itemize}
\item $\varphi^f_h:HOM(G,H)\to HOM(F,K)$  by $\varphi^f_h(x)=hxf$, and
\item $\varphi^g_k:HOM(F,K)\to HOM(G,H)$  by $\varphi^g_k(y)=kyg$.
\end{itemize}

In this section, we include proofs of several standard lemmas for completeness. In many cases, different forms of essentially the same results can 
be found in work such as~\cite{Dochtermann,Fieux,HellNesetril}. The next lemma is similar to~\cite[Chapter 2 Exercise \# 8 (b)]{HellNesetril}. Following~\cite{BondyMurty} we use $c(X)$ to denote the number of components of a graph $X$.

\begin{lem}
\label{homs}
The mappings
\[\varphi^f_h:\mathscr{H}_H(G)\to\mathscr{H}_K(F)\mbox{ and }\]
\[\varphi^g_k:\mathscr{H}_K(F)\to\mathscr{H}_H(G)\]
are graph homomorphisms.
\end{lem}

\begin{proof}
Let $y$ and $y'$ be a pair of adjacent vertices in $\mathscr{H}_K(F)$ 
and let $uv$ be an edge of $G$. Then, 
$\varphi^g_k(y)(u)\varphi^g_k(y')(v)=k y g(u) ky' g(v)$ is an edge of 
$H$ because $y \sim y'$ and $g,k$ are homomorphisms. Therefore
$\varphi^g_k(y)\sim\varphi^g_k(y')$ and so $\varphi^g_k$ is a homomorphism.
By symmetry, $\varphi^f_h$ is a homomorphism as well.
\end{proof}

\begin{lem}
\label{retract}
Suppose that $g$ and $h$ are retractions with corresponding sections $f$ 
and $k$. Then $\varphi^f_h:\mathscr{H}_H(G)\to\mathscr{H}_K(F)$ is a
retraction, with $\varphi^g_k$ as the corresponding section.
\end{lem}

\begin{proof}
Given $y:F\to K$, we have that $\varphi^f_h\varphi^g_k (y) = hk y gf = y$ since 
$gf$ and $hk$ are the identities on $F$ and $K$ respectively.  
\end{proof}

\begin{cor}
\label{retracts} 
If $G$ retracts to $F$ and $H$ retracts to $K$, then 
$c(\mathscr{H}_H(G)) \geq c(\mathscr{H}_K(F))$.
\end{cor}

\begin{proof}
By Lemma~\ref{retract}, $\mathscr{H}_H(G)$ retracts to
$\mathscr{H}_K(F)$. Retractions are surjective.  Homomorphisms map
components to components. The result follows.
\end{proof}

The next result was proven by Fieux and Lacaze in~\cite{Fieux}. We include a 
proof here for completeness.

\begin{lem}[Fieux and Lacaze~\cite{Fieux}]
\label{fold}
Suppose that $g$ and $h$ are dismantling retractions with corresponding sections $f$ 
and $k$. Then $\varphi^f_h:\mathscr{H}_H(G)\to\mathscr{H}_K(F)$ is a dismantling
retraction, with $\varphi^g_k$ as the corresponding section.
\end{lem}

\begin{proof}
We suppose that $g$ and $h$ are folds or isomorphisms, from
which the general case follows. By Lemma~\ref{retract} we know that $\varphi_h^f:\mathscr{H}_H(G)\to \mathscr{H}_K(F)$
is a retraction with corresponding section $\varphi_k^g$. To show
that $\varphi_h^f$ is a dismantling retraction, we proceed in two steps:
first we show that $\mathscr{H}_H(G)$ dismantles to $\mathscr{H}_H(F)$, and
secondly we show $\mathscr{H}_H(F)$ dismantles to $\mathscr{H}_K(F)$.

Let $g : G \to F$ be a fold that maps $v' \to v$ and fixes all other elements of $G - v'$.
Partition the elements of $HOM(G,H)$ into sets $A = \{ x \in HOM(G,H) | x(v) = x(v') \}$ and
$B =  \{ x \in HOM(G,H) | x(v) \neq x(v') \}$. We claim 
$\mathscr{H}_H(F)$ is isomorphic to the subgraph of $\mathscr{H}_H(G)$ induced by $A$.  Given
$y \in HOM (F,H)$, define $\varphi_1^g(y) = yg$.  Similarly given $x \in A$
define $\varphi_1^f(x) = xf$.  Consider the composition $\varphi_1^f\varphi_1^g(y) = ygf = y$ since
$gf$ is the identity.  Conversely the composition $\varphi_1^g\varphi_1^f(x) = xfg = x$ when
restricted to $G-v'$ since $g$ is a fold.  However $xfg(v') = x(v) = x(v')$ since $x \in A$.
Thus $\varphi_1^g\varphi_1^f$ is the identity on $A$.  Consequently $\varphi_1^f$ and $\varphi_1^g$
are both bijections which establishes the isomorphism claim.

We now show $\mathscr{H}_H(G)$ dismantles to the subgraph induced by $A$
which is equivalent to dismantling to $\mathscr{H}_H(F)$.  Let $x' : G \to H$
and define $x \in A$ by
$$
x(u) = \left\{ \begin{array}{ll}
x'(u) & \mbox{ if } u \neq v' \\
x(v) & \mbox{ if } u = v'
\end{array} \right.
$$
Let $z : G \to H$ such that $z \sim x'$.  Let $wu \in E(G)$.
If $u \neq v'$, then $x(u) = x'(u)$.  Hence $z(w)x'(u) = z(w)x(u)$.
Since $z(w)x'(u) \in E(H)$, we conclude $z(w)x(u) \in E(H)$.
On the other hand, if $u = v'$, then 
$z(w)x(v') = z(w)x(v) = z(w)x'(v)$ and $z(w)x'(v) \in E(H)$.
We conclude $z \sim x$.
Hence, the mapping $x' \mapsto x$ is a fold of $\mathscr{H}_H(G)$ to $\mathscr{H}_H(G) - x'$.
Continuing in this manner, we may dismantle $\mathscr{H}_H(G)$ to $\mathscr{H}_H(F)$.

Using an analogous argument, one may show $\mathscr{H}_K(F)$ is isomorphic
to an induced subgraph of $\mathscr{H}_H(F)$, and that the latter dismantles 
to the former using the maps $\varphi_h^1$ and $\varphi_k^1$.
%
%
%
\end{proof}

\begin{cor}
\label{folds} 
If $G$ has a dismantling retraction to $F$ and $H$ has dismantling retraction to $K$, then 
$c(\mathscr{H}_H(G)) = c(\mathscr{H}_K(F))$.
\end{cor}

\begin{proof}
By Lemma~\ref{fold} there is a dismantling retraction $\varphi^f_h:\mathscr{H}_H(G)\to\mathscr{H}_K(F)$ with section 
$\varphi^g_k:\mathscr{H}_K(F)\to\mathscr{H}_H(G)$. 
For $x:G\to H$, $N(x)\subseteq N(\varphi^g_k\varphi^f_h(x))$. Since $\mathscr{H}_H(G)$ is reflexive, this implies that 
$x\sim \varphi^g_k\varphi^f_h(x)$. So, the map $\varphi^g_k\varphi^f_h$ maps each component of $\mathscr{H}_H(G)$ to itself. 
In particular, $\varphi^f_h$ must act injectively on the components of $\mathscr{H}_H(G)$.  Thus,
$c(\mathscr{H}_H(G)) \leq c(\mathscr{H}_K(F))$.  Conversely, Corollary~\ref{retracts} gives
$c(\mathscr{H}_H(G)) \geq c(\mathscr{H}_K(F))$.
\end{proof}

We can now calculate the circular mixing number and threshold for bipartite graphs related to $K_2$. 

\begin{prop}
\label{K_mmtrees}
Let $G$ be a tree or a complete bipartite graph on at least two vertices. Then $\mathfrak{M}_c(G)=\mathfrak{m}_c(G)=2$.
\end{prop}

\begin{proof}
For any such graph $G$ there is a dismantling retraction to $K_2$, and so it suffices to show that 
$\mathfrak{M}_c(K_2)=\mathfrak{m}_c(K_2)=2$ by Corollary~\ref{folds}. However, it is an easy exercise to prove that 
$K_2$ is $(k,q)$-mixing if and only if $k/q>2$. The result follows.
\end{proof}

\subsection{Self mixing graphs}

In this section we examine the mixing properties of the endomorphism monoid, i.e. the collection of
homomorphisms $f: G \to G$.  We begin with general observations about the flexibility to locally modify a homomorphism.
The degree of a vertex $f$ in $\mathscr{H}_H(G)$ provides some indication of this flexibility. 
In this way, the following definition describes the most inflexible homomorphisms.

\begin{defn}
For a vertex $v\in V(G)$, we say that a homomorphism $\varphi:G\to H$ is  
\emph{frozen} if it is an isolated (reflexive) vertex in $\mathscr{H}_H(G)$.
\end{defn}

Slightly different definitions of frozen homomorphisms and colourings are given in~\cite{Winkler,Connectedness}.
In~\cite{BrewsterNoel} we study extensions of circular colourings and use frozen homomorphisms 
in several examples. (We used the terminology \emph{uniquely extendible} homomorphisms to fit the theme
of that paper.) The following is a straightforward observation, which has been noted by several authors 
including~\cite{Belanger,Dochtermann,Fieux}.

\begin{lem}[B{\'e}langer et al.~\cite{Belanger}; Dochtermann~\cite{Dochtermann}; Fieux and Lacaze~\cite{Fieux}]
\label{id}
$F$ is stiff if and only if every automorphism $\varphi:F\to F$ is frozen.
\end{lem}

\begin{proof}
Suppose that there is an automorphism $\varphi:F\to F$ which is not frozen. 
Thus there exists $\psi: F \to F$ and $v \in V(F)$ such that $\varphi \sim \psi$
and $\varphi(v) \neq \psi(v)$.  We have $N(\varphi(v)) = \varphi(N(v))$ since
$\varphi$ is an automorphism, and for every edge $uv$ we have $\varphi(u) \sim \psi(v)$
since $\varphi \sim \psi$.
Putting this together, we have $N(\varphi(v))=\varphi(N(v))\subseteq N(\psi(v))$. 
Hence $F$ is not stiff because we can define the fold $f:V(F) \to V(F)$ by 
\[
f(u)=\left\{\begin{array}{ll}
	u & \mbox{ if } u\neq \varphi(v),\\ 
	\psi(v) & \mbox{ if } u = \varphi(v).
	\end{array}
\right.
\]

On the other hand, if $F$ is not stiff, then the identity $1_F$ on $F$ is adjacent to an
arbitrary fold on $F$, which implies that $1_F$ is not frozen. The result follows.
\end{proof}

We can now characterize graphs $G$ such that $\mathscr{H}_G(G)$ is connected.
We remark this characterization is similar to the result of Brightwell and 
Winkler~\cite{Winkler}; however, their notion of dismantlable differs slightly
from ours (for non-reflexive graphs).

\begin{thm}
\label{TFAE}
Let $G$ be a graph and $G^*=\mathscr{H}_G(G)$. The following are equivalent:
\begin{enumerate}[(a)]
\item $\mathscr{H}_G(G)$ is connected; \label{GG}
\item $G$ is dismantlable;\label{Gd}
\item $\mathscr{H}_{G^*}(G^*)$ is connected;\label{HGG}
\item $G^*$ is dismantlable.\label{HGGd}
\end{enumerate}
\end{thm}

\begin{proof}
First let us show (\ref{GG}) $\Leftrightarrow$ (\ref{Gd}). 
Recall by Theorem~\ref{uniqueCoreStiff} (\ref{uniqueCore}) there is a dismantling retraction (a sequence of folds) $g:G\to F$ where $F$ is stiff. 
If $G$ is not dismantlable, then by definition $F$ is not rigid. Let $\psi:F\to F$ so that $\psi\neq1_F$, where $1_F$ denotes the identity on $F$. 
By Lemma~\ref{id}, the map $1_F$ is frozen and so $\psi$ is not $\times$-homotopic to $1_F$. It follows that $\mathscr{H}_F(F)$ 
is disconnected, and thus $\mathscr{H}_G(G)$ is disconnected by Corollary~\ref{folds}. On the other hand, if $G$ is dismantlable, then 
$F$ is rigid.  Therefore $\mathscr{H}_F(F)$ is composed of a single reflexive vertex, 
and thus $\mathscr{H}_G(G)$ is connected by Corollary~\ref{folds}. The same argument proves (\ref{HGG}) $\Leftrightarrow$ (\ref{HGGd}) 
by replacing $G$ with $G^*$. 

Next we show (\ref{Gd}) $\Rightarrow$ (\ref{HGGd}). Indeed, let $g:G\to F$ be a dismantling retraction such that $F$ is rigid. Then by Lemma~\ref{fold} there is a dismantling retraction $\varphi:\mathscr{H}_G(G)\to\mathscr{H}_F(F)$. Since $F$ is rigid we have that $\mathscr{H}_F(F)$ is composed of a single reflexive vertex, and is therefore rigid. Therefore $G^*$ is dismantlable.

Finally, note that the only rigid graph to which a reflexive graph can dismantle is a single reflexive vertex,
and no fold can destroy a component.  Thus every reflexive dismantable graph is connected.  We conclude,
(\ref{HGGd}) $\Rightarrow$ (\ref{GG}) since $G^*=\mathscr{H}_G(G)$. The result follows.
\end{proof}

This translates into the following result for loop-free graphs.

\begin{cor}
A loop-free graph $G$ is $G$-mixing if and only if $G$ is dismantlable.
\end{cor}

\begin{proof}
By Theorem~\ref{TFAE}, $\mathscr{H}_G(G)$ is connected if and only if $G$ is dismantlable. By Proposition~\ref{colourHom} $\mathscr{C}_G(G)$ and $\mathscr{H}_G(G)$ have the same components and so $G$ is $G$-mixing if and only if $G$ is dismantlable. 
\end{proof}

%
%
\section{Homotopies of non-surjective {$\boldsymbol{(k,q)}$}-colourings}
\label{parents}

The goal of this section is to apply the ideas from Section~\ref{MixHom} to (i)
obtain upper bounds on $\mathfrak{M}_c$ which are tighter than 
Theorem~\ref{col(G)Bound} for certain graphs, and (ii) provide a condition on $G$ which implies $\mathfrak{M}_c(G)\leq \mathfrak{M}(G)$. To achieve this goal, we focus on
$(k,q)$-colourings which are $\times$-homotopic to non-surjective 
colourings. 

\begin{defn}
Suppose that $G$ is $(k,q)$-colourable. If every $(k,q)$-colouring of $G$ is $\times$-homotopic to a non-surjective 
$(k,q)$-colouring, we say that $G$ is \emph{$(k,q)$-flexible}.
\end{defn}

We begin with the case that $\gcd(k,q)=1$ and focus on the relationship between $(k,q)$ and its lower parent $(k',q')$. 

\begin{prop}[Bondy and Hell~\cite{BondyHell}; Hell and Ne{\v{s}}et{\v{r}}il~\cite{HellNesetril}]
\label{d=1}
Let $(k',q')$ be the lower parent of $(k,q)$. Then for $0\leq i\leq k-1$ there is a dismantling retraction $r:(G_{k,q}-i)\to G_{k',q'}$ 
\end{prop}

\begin{lem}
\label{relpr}
Let $(k',q')$ be the lower parent of $(k,q)$. If $G$ is $(k,q)$-flexible, then $c(\mathscr{H}_{k,q}(G)) \leq c(\mathscr{H}_{k',q'}(G))$. 
\end{lem}

\begin{proof}
Let $S$ be the set of all $(k,q)$-colourings of $G$ which do not map to the colour $0$. It is clear that the subgraph $X$ of $\mathscr{H}_{k,q}(G)$ induced by $S$ is isomorphic to $\mathscr{H}_{G_{k,q}-0}(G)$. By Proposition~\ref{d=1} and Corollary~\ref{folds}
$c(X) = c(\mathscr{H}_{k',q'}(G))$. Therefore, it suffices to show that every $(k,q)$-colouring of $G$ is $\times$-homotopic 
to at least one element of $S$. 
(In the case that some $(k,q)$-colouring is homotopic to elements of $S$ in different components of $X$, $\mathscr{H}_{k,q}(G)$ will
have fewer components than $\mathscr{H}_{k',q'}(G)$.)

Let $f:G\to G_{k,q}$ be arbitrary. Since $G$ is $(k,q)$-flexible, $f$ is $\times$-homotopic to a non-surjective homomorphism $g$. If $g$ does not map to $0$, then we are done as $g\in S$. Otherwise, suppose that $g$ does not map to colour $i$. Notice that every vertex coloured $i+q$ by $g$ can be recoloured $i+q-1$. This gives us a $(k,q)$-colouring $h$ which is $\times$-homotopic to $g$ (and therefore $\times$-homotopic to $f$) and does not use $i+q$.   Continuing in this way we can construct a sequence of colourings that do not use
colour $i+2q, i+3q, \dots$ (respectively), and are each $\times$-homotopic to $f$.  Since $k$ and $q$ are relatively prime, we will reach an element of $S$ as required.
\end{proof}

\begin{prop}
\label{ret}
Let $k,q,d$ be positive integers such that $\gcd(k,q)=1$. Then there is a retraction $r:G_{kd,qd}\to G_{k,q}$. 
\end{prop}

\begin{proof}
Define $r:G_{kd,qd}\to G_{k,q}$ by $r(u) = \lfloor u/d \rfloor$.  Suppose $qd \leq u-v \leq kd-qd$.  
It is straightforward to verify $q \leq \lfloor u/d \rfloor - \lfloor v/d \rfloor \leq k-q$.
\end{proof}

\begin{prop}
\label{d>1}
Let $k,q,d$ be positive integers such that $\gcd(k,q)=1$. Then for $0\leq i\leq kd-1$ 
there is a dismantling retracion $r:(G_{kd,qd}-i)\to G_{k(d-1),q(d-1)}$.
\end{prop}

\begin{proof}
As in the proof Lemma~\ref{relpr}, we can map $(G_{kd,qd}-i) \to (G_{kq,qd}-\{i, i+dq \})$.  In fact this mapping is a fold.  
Continuing we can dismantle $G_{kq,dq}-i$ to $G_{kq,dq} - \{ i, i+dq, i+2dq, \dots, i+(k-1)dq \}$.  One can verify this latter graph 
is isomorphic to $G_{k(d-1),q(d-1)}$.
\end{proof}

\begin{lem}
\label{nonrelpr}
Let $k,q,d$ be positive integers such that $\gcd(k,q)=1$. If $G$ is $(kt,qt)$-flexible for all $t$, $1\leq t\leq d$, then 
$c(\mathscr{H}_{kd,qd}(G))=c(\mathscr{H}_{k,q}(G))$.
\end{lem}

\begin{proof}
First, by Proposition~\ref{ret} we have that $G_{kd,qd}$ retracts to $G_{k,q}$ and so $c(\mathscr{H}_{kd,qd}(G)) \geq c(\mathscr{H}_{k,q}(G))$ by Corollary~\ref{retracts}. For the reverse inequality, we will apply induction on $d$ where the case for $d=1$ is trivial. Now suppose $d\geq2$ and let $S$ be the set of all $(kd,qd)$-colourings of $G$ which do not map to the colour $0$. It is clear that the subgraph $X$ of $\mathscr{H}_{kd,qd}(G)$ induced by $S$ is isomorphic to $\mathscr{H}_{G_{kd,qd}-0}(G)$. By Proposition~\ref{d>1} and Corollary~\ref{folds}, 
$c(X) = c(\mathscr{H}_{k(d-1),q(d-1)}(G))$. By induction, $c(X)=c(\mathscr{H}_{k,q}(G))$. Therefore, in order to show that 
$c(\mathscr{H}_{kd,qd}(G)) \leq c(\mathscr{H}_{k,q}(G))$ it will be enough to show that every $(kd,qd)$-colouring of $G$ is $\times$-homotopic to an element of $S$. Let us first prove the following:

\begin{claim}
\label{<k}
Every $(kd,qd)$-colouring $f$ of $G$ is $\times$-homotopic to a $(kd,qd)$-colouring $g$ such that $|g(V(G))|<k$. 
\end{claim}

The proof of the claim is also by induction on $d$. The case for $d=1$ simply follows by the fact that $G$ is $(k,q)$-flexible. Now suppose $d\geq2$ and let $f:G\to G_{kd,qd}$ be arbitrary. Since $G$ is $(kd,qd)$-flexible, $f$ is $\times$-homotopic to a $(kd,qd)$-colouring $g$ which is non-surjective. Let $i$ be a colour which is not mapped to by $g$. As above, we construct a colouring $f_{d-1}$ which does not use colours from 
$C^i=\{i,i+qd,\dots, i+(k-1)qd\}$. 

As in Proposition~\ref{d>1}, the graph $H_{d-1}=G_{kd,qd}-C^i$ is isomorphic to $G_{k(d-1),q(d-1)}$. Let $f'$ be the homomorphism obtained by restricting the codomain of $f_{d-1}$ to $H_{d-1}$. By the inductive hypothesis $f'$ is $\times$-homotopic to a $H_{d-1}$-colouring $g'$ such that $|g'(V(G))|<k$. Let $\iota:H_{d-1}\to G_{kd,qd}$ be the inclusion map. It is straightforward to show that $\iota f'$ is $\times$-homotopic to $\iota g'$. Finally, we have $\iota f'=f_{d-1}$ and $|\iota g'(V(G))|<k$. This concludes the proof of Claim~\ref{<k}.

To finish the proof of the lemma, let $f:G\to G_{kd,qd}$ be arbitrary and let $f$ be $\times$-homotopic to $g$ where $|g(V(G))|<k$. Then, in particular, there is some $j\in C^0=\{0,qd,\dots,(k-1)qd\}$ such that $g$ does not map to $j$. As above we construct a sequence of colourings,
each $\times$-homotopic to $g$ that do not use colours $j+qd, j+2qd, j+3qd, \dots$.  Eventually, we will reach a $(kd,qd)$-colouring $h$ which does not use $0$, i.e. $h\in S$. The result follows.
\end{proof}

Consider the following straightforward facts about lower parents.

\begin{prop}
\label{q<=b}
Suppose $k$ and $q$ are relatively prime positive integers, and let $(k',q')$ be the lower
parent of $(k,q)$.  If $k/q > j/p$ for some positive integers $j$ and $p$, then
$$
\frac{k'}{q'} \geq \frac{j}{p} + \frac{q'-p}{pqq'}.
$$
In particular $k'/q' \geq j/p$ if $q'=p$ and $k'/q' > j/p$ if $q' > p$.
\end{prop}

\begin{proof}
Observe,
\begin{eqnarray*}
k/q & > & j/p \\
kp/q & > & j \\
kp/q & \geq & j + 1/q \\
k/q & \geq & j/p + 1/(pq) \\
k'/q' = k/q - 1/(qq') & \geq & j/p + 1/(pq) - 1/(qq') = j/p + (q'-p)/(pqq')
\end{eqnarray*}
\end{proof}
%
%
%

Next, we use Lemmas~\ref{relpr} and~\ref{nonrelpr} to prove some general upper bounds on $\mathfrak{M}_c(G)$. Note that the following result improves on Theorem~\ref{col(G)Bound} for the case that $G$ has a $\Delta$-regular component.

\begin{lem}
\label{reduceDelta}
Let $G$ be a graph with at least one edge and $\frac{k}{q} > 2 \Delta(G)$. Then $G$ is $(kd,qd)$-flexible for all $d\geq1$. 
\end{lem}

\begin{proof}
Let $f:G\to G_{kd,qd}$ be arbitrary. Notice that $kd \geq 2dq\Delta(G)+1\geq (2dq-1)\Delta(G)+2$ since $G$ contains an edge. 
By Eq.~\eqref{eqn:availcirc} (of Section~\ref{degen}), for each $v\in V(G)$ there are at least two available colours for $v$.  Therefore, we may perform 
a local modification which recolours each vertex of $f^{-1}(\{0\})$ with a colour different from $0$. The result is a non-surjective 
$(kd,qd)$-colouring $g$ which is adjacent to $f$ in $\mathscr{H}_{k,q}(G)$. 
\end{proof}

\begin{thm}
\label{2Delta}
If $G$ is a graph with at least one edge, then $\mathfrak{M}_c(G)\leq2\Delta(G)$.
\end{thm}

\begin{proof}
Suppose that $G$ contains an edge and let $\frac{k}{q}>2\Delta(G)$ where $\gcd(k,q)=1$. If $col(G)=2$,
then $G$ is a forest and $G$ is $(k,q)$-mixing by Propotion~\ref{K_mmtrees}. So, we may suppose $col(G)\geq3$.

We show by induction on $q$ that $G$ is $(k,q)$-mixing, from which it follows that $G$ is $(kd,qd)$-mixing for all $d\geq1$ 
by Lemmas~\ref{nonrelpr} and~\ref{reduceDelta}. 
In the base case $q=1$, recall that $col(G)\geq3$ and so we have $k> 2\Delta(G)\geq 2(col(G)-1) \geq col(G)+1$. 
Therefore $G$ is $k$-mixing by Theorem~\ref{col}. Now suppose $q\geq2$ and let $(k',q')$ be the lower
parent of $(k,q)$. We will be done by Lemmas~\ref{relpr} and~\ref{reduceDelta} if we can show that $G$ is $(k',q')$-mixing.

Suppose that $q'=1$. In this case, we have that $\frac{k}{q}>2\Delta(G)$ implies $k'\geq 2\Delta(G)\geq 2(col(G)-1) \geq col(G)+1$ 
by Proposition~\ref{q<=b} and the fact that $col(G)\geq3$. (Let $j = 2\Delta(G)$ and $p=1$.) 
Therefore $G$ is $(k',q')$-mixing by Theorem~\ref{col}. On the other hand, 
if $q'>1$, then $q' > p=1$. 
Hence $\frac{k}{q}>2\Delta(G)$ implies $\frac{k'}{q'}> 2\Delta(G)$ by Proposition~\ref{q<=b}. 
By definition of lower parent $q' \leq q$.  In this case, equality is impossible, i.e. $q' < q$, 
otherwise $q'$ divides $kq'-k'q$ contradicting $kq'-k'q=1$.
Thus, $G$ is $(k',q')$-mixing by the inductive hypothesis. The result follows. 
\end{proof}

We now bound the circular mixing threshold in terms of the mixing threshold. 
Given the fact that $\chi_c\leq \chi$, one might expect that $\mathfrak{M}_c\leq \mathfrak{M}$ in general.  However the situation is more complicated. 
In~\cite{BrewsterNoel} we provide examples where the circular
mixing threshold exceeds the mixing threshold.  For example, by Theorem~4 in~\cite{BrewsterNoel} there is a homomorphism
$\varphi: G_{19,7} \to G_{19,2}$ that is frozen which implies $\mathfrak{M}_c(G_{19,7}) \geq 19/2$.  However, 
$col(G_{19,7}) + 1 = 7 \geq \mathfrak{M}(G_{19,7})$.  Thus, our bound requires a second term, namely $(|V|+1)/2$.

\begin{lem}
\label{v+1/2flex}
Let $G$ be a graph and let $\frac{k}{q}>\max\left\{\mathfrak{M}(G),\frac{|V(G)|+1}{2}\right\}$.
Then $G$ is $(kd,qd)$-flexible for all $d\geq1$. 
\end{lem}

\begin{proof}
If $qd>1$, then we have $kd>|V(G)|+1$ since $\frac{kd}{qd}>\frac{|V(G)|+1}{2}$. In this case, no $(kd,qd)$-colouring
can be surjective, and so we are done. Now, if $q=1$ and $d=1$ we have that $k>\mathfrak{M}(G)\geq\chi(G)$ and so there
is a non-surjective $k$-colouring of $G$. Moreover, $G$ is $k$-mixing and so every $k$-colouring is $\times$-homotopic
to the non-surjective colouring.
\end{proof}

\begin{thm}
\label{v+1/2}
For a graph $G$,
\[\mathfrak{M}_c(G)\leq\max\left\{\frac{|V(G)|+1}{2},\mathfrak{M}(G)\right\}.\]
\end{thm}

\begin{proof}
Suppose $\frac{k}{q}>\max\left\{\mathfrak{M}(G),\frac{|V(G)|+1}{2}\right\}$. 
First we show that $G$ is $(k,q)$-mixing by induction on $q$.
It will follow that $G$ is $(kd,qd)$-mixing for all $d\geq1$ by Lemmas~\ref{nonrelpr} and~\ref{v+1/2flex}. 
The case for $q=1$ is trivial as in this case we have $k>\mathfrak{M}(G)$.  Now, suppose that $q\geq2$ and let $(k',q')$ be the lower parent of 
$(k,q)$. We will be done by Lemmas~\ref{relpr} and~\ref{v+1/2flex} if we can show that $G$ is $(k',q')$-mixing.

\begin{case}
Suppose that $q'=1$. 
\end{case}

Then $\frac{k}{q}>\mathfrak{M}(G)$ implies $k'\geq \mathfrak{M}(G)$ by Proposition~\ref{q<=b}. It follows that $G$ is $k'$-mixing. 

\begin{case}
Suppose that $q'=2$. 
\end{case}

Then $\frac{k}{q}>\mathfrak{M}(G)$ implies $\frac{k'}{2}> \mathfrak{M}(G)$ by Proposition~\ref{q<=b}. Also, by Proposition~\ref{q<=b} we have that $\frac{k}{q}>\frac{|V(G)|+1}{2}$ implies $\frac{k'}{2}\geq\frac{|V(G)|+1}{2}$. Therefore $G$ is $(k',2)$-flexible since no $(k',2)$-colouring of $G$ can be surjective. Now, let $k''=\frac{k'-1}{2}$. It is easy to check that $(k'',1)$ is the lower parent of $(k',2)$ and $k''\geq\mathfrak{M}(G)$ by Proposition~\ref{q<=b}. Therefore $G$ is $k''$-mixing, which implies that $G$ is $(k',2)$-mixing by Lemma~\ref{relpr}. 

\begin{case}
Suppose that $q'\geq3$. 
\end{case}

\addtocounter{case}{-3}

In this case, we have $\frac{k'}{q'}>\max\left\{\frac{|V(G)|+1}{2},\mathfrak{M}(G)\right\}$ by Proposition~\ref{q<=b}. Therefore $G$ is $(k',q')$-mixing by the inductive hypothesis. The result follows. 
\end{proof}

We obtain the following corollary.

\begin{cor}
\label{2}
If $|V(G)|\leq 2\mathfrak{M}(G)-1$, then $\mathfrak{M}_c(G)\leq\mathfrak{M}(G)$. 
\end{cor}

\section{A lower bound}
\label{lb}
In this section we establish a lower bound on $\mathfrak{m}_c$ for non-bipartite graphs. Specifically, we prove the following bound based on the clique number.

\begin{thm}
\label{omega}
If $G$ is a non-bipartite graph, then $\mathfrak{m}_c(G)\geq\max\{4,\omega(G)+1\}$.
\end{thm}

By Definition~\ref{defn:mt}, $\mathfrak{m}_c \geq \chi_c > \chi -1$.  If $\chi > \omega + 1$, then $\chi-1 \geq \omega +1$ and 
$\mathfrak{m}_c > \omega + 1$.  Thus, this result is of interest for graphs with $\chi \leq \omega + 1$ which includes odd cycles and
perfect graphs, particularly cliques.

Theorem~\ref{omega} generalizes a result of~\cite{Connectedness} which asserts that no $3$-chromatic graph $G$ is $3$-mixing. Their proof focuses on $3$-colourings of an odd cycle within $G$. They classify $3$-colourings of $G$ based on certain configurations realized on the odd cycle, and argue that colourings corresponding to different configurations must lie in different components of $\mathscr{C}_3(G)$. We follow a similar trajectory, except that we deal with circular colourings of both odd cycles and cliques. First, we need a general notion to describe colourings with limited flexibility.  Recall Eq.~\eqref{eqn:availcirc} (of Section~\ref{degen}) where we find the set of available colours for recolouring the vertex $v$.  We 
now examine the case where this set is a contiguous interval of colours in $G_{k,q}$.

\begin{defn}
A $(k,q)$-colouring $f$ of a graph $F$ is said to be \emph{constricting} if, for every $v\in V(F)$, the set $\bigcap_{u\in N(v)}N(f(u))$ of available colours for $v$ is an interval of $V(G_{k,q})$. 
\end{defn}

Notice, for example, that every frozen colouring is constricting. As it turns out, every $(k,q)$-colouring is constricting when $k/q<4$. Also, in the case of an $r$-clique, every $(k,q)$-colouring is constricting when $k/q<r+1$.

\begin{lem}
\label{<4}
If $k/q<4$, then every  $(k,q)$-colouring $f$ of a graph $F$ is constricting.
\end{lem}

\begin{proof}
Let $v\in V(F)$. Recall Eq.~\eqref{eqn:availcirc}: the set of available colours for $v$ is
\[\bigcap_{u\in N(v)}N(f(u))=V(G_{k,q}) - \bigcup_{u\in N(v)}[f(u)-q+1,f(u)+q-1].\]
For each $u\in N(v)$, define $[a_u,b_u]=[f(u)-q+1,f(u)+q-1]$. If the set of available colours for $v$ is not an interval, then there must exist 
$u,w\in N(v)$ such that 
\[[a_u,b_u]\cap[a_w,b_w]=\emptyset.\]
Also, there must be at least two available colours for $v$, say $f(v)$ and $c$. Therefore,
\[k=|V(G_{k,q})|\geq |\{f(v),c\}|+|[a_u,b_u]|+|[a_w,b_w]|\]
\[=2+(2q-1)+(2q-1)=4q\]
contradicting the assumption that $k/q<4$. The result follows. 
\end{proof}

\begin{lem}
\label{<m+1}
If $k/q<r+1$, then every $(k,q)$-colouring $f$ of $K_r$ is constricting.
\end{lem}

\begin{proof}
Let $v_0\in V(K_r)$ be arbitrary. Suppose, without loss of generality, that $f(v_0)=0$. Label the vertices of $V(K_r)-v_0$ by $v_1,v_2,\dots,v_{r-1}$ so that
\[0=f(v_0)<f(v_1)<\dots<f(v_{r-1})\leq k-q.\]
For $1\leq i\leq r-1$ define $[a_i,b_i]=[f(v_i)-q+1,f(v_i)+q-1]$. For each $i$, $1\leq i\leq r-2$, we must have $f(v_i)+q\leq f(v_{i+1})$. This implies that for any $1\leq s<t\leq r-1$,
\begin{equation}\label{intervals}\left|\bigcup_{i=s}^t[a_i,b_i]\right|\geq (2q-1)+(t-s)q = (t-s+2)q-1.\end{equation}

If the set of available colours for $v_0$ is not an interval, then there are at least two colours, say $f(v_0)=0$ and $c$, 
which are available for $v_0$ and for which some $i'$ satisfies $1\leq i'\leq r-2$ and $b_{i'}<c<a_{i'+1}$. This implies
\begin{equation}\label{emptycap}\left(\bigcup_{i=1}^{i'}[a_i,b_i]\right)\bigcap \left(\bigcup_{i=i'+1}^{r-1}[a_i,b_i]\right) =\emptyset.\end{equation}
Therefore,
\[k=|V(G_{k,q})|\geq |\{0,c\}|+\left|\bigcup_{i=1}^{r-1}[a_i,b_i]\right|\]
\[\geq 2+ ((i'+1)q-1) + ((r-i')q-1) = (r+1)q\]
by Eq.~\eqref{intervals} and Eq.~\eqref{emptycap}. This contradicts the assumption that $k/q<r+1$. The result follows. 
\end{proof}

We are interested in the behaviour of constricting colourings on cycles when moving to adjacent colourings in the colour graph. 
To make this precise, consider the following definition.  (Note in the following indices of vertices in $C$ are reduced modulo $\ell + 1$.)

\begin{defn}
\label{tausigma}
Suppose that $C=(v_0,v_1,\dots,v_\ell)$ is a cycle of length $\ell+1$ in $F$. Given a $(k,q)$-colouring $f$ of $F$ and $0\leq i\leq \ell$, define
\[\tau(f,i):= f(v_{i+1})-f(v_i)\mod k.\]
Also, define
\[\sigma(f):=\sum_{i=0}^\ell \tau(f,i).\]
\end{defn}

The value of $\sigma(f)$ provides information about how many times the sequence $f(v_0),f(v_1),\dots,f(v_\ell),f(v_0)$ 
wraps around $V(G_{k,q})$. For constricting colourings, $\sigma$ is invariant under moving to adjacent colourings in the colour graph. 

\begin{prop}
\label{samesigma}
Let $f$ be a constricting $(k,q)$-colouring of $F$. If $f\sim g$ in $\mathscr{C}_{k,q}(F)$, then $\sigma(g)=\sigma(f)$. 
\end{prop}

\begin{proof}
Note that $f$ and $g$ differ on exactly one vertex. If that vertex is not on $C$, then $\sigma(g)=\sigma(f)$ trivially. 
So, suppose, without loss of generality, that $g(v_0)\neq f(v_0)$. This means that the set of available colours for $v_0$ under $f$ 
is an interval contained in $[f(v_\ell)+q,f(v_1)-q]$. In particular, $g(v_0)$ lies in this interval. It follows that 
$\tau(f,\ell)+\tau(f,0)=\tau(g,\ell)+\tau(g,0)$, which proves the result. 
\end{proof}

Thus, in the case that all $(k,q)$-colourings of $G$ are constricting, colourings with different values of $\sigma$ must lie in 
different components of the colour graph. We can show that $G$ is not $(k,q)$-mixing by simply exhibiting $(k,q)$-colourings which 
attain different values of $\sigma$. Sometimes this can be done by composing a $(k,q)$-colouring $f$ with a reflection of $V(G_{k,q})$ about $0$. 

\begin{prop}
\label{sigma+sigma}
Let $f$ be a $(k,q)$-colouring of $F$ and define $f'$ by $f'(v)=k-f(v)\mod k$ for all $v\in V(F)$. Then $f'$ is a $(k,q)$-colouring of $F$ and  $\sigma(f)+\sigma(f')=(\ell+1)k$. 
\end{prop}

\begin{proof}
The fact that $f'$ is a $(k,q)$-colouring of $F$ follows easily from the definition of a $(k,q)$-colouring. For $0\leq i\leq \ell$ we have 
\[q\leq\tau(f,i)\leq k-q,\text{ and}\]
\[q\leq\tau(f',i)\leq k-q.\]
In particular, $0<\tau(f,i)+\tau(f',i)<2k$. Now, notice that $\tau(f,i)+\tau(f',i)=0\mod k$, and so the only possiblility is that $\tau(f,i)+\tau(f',i)=k$. Summing over all $i$ gives the desired result. 
\end{proof}

\begin{prop}
\label{kdividessigma}
If $f$ is a $(k,q)$-colouring of $F$, then $k\mid \sigma(f)$. 
\end{prop}

\begin{proof}
Recall that for $0\leq i\leq \ell$ we have $\tau(f,i)=f(v_{i+1})-f(v_i)\mod k$.  Thus, $\tau(f,i) = f(v_{i+1}) - f(v_i) + q_i \cdot k$.
The sum $\sum_{i=0}^\ell \tau(f,i) = \sigma(f)$ reduces to $\sum q_i \cdot k$. The result follows. 
\end{proof}

We are now in position to prove Theorem~\ref{omega}.

\begin{proof}[Proof of Theorem~\ref{omega}]
Let $G$ be a non-bipartite graph. If $\omega(G)\geq 4$, then let $F$ be an $\omega(G)$-clique in $G$. Otherwise, let $F$ be an odd cycle in $G$. Let $k/q<\max\{4,\omega(G)+1\}$ and let $g$ be any $(k,q)$-colouring of $G$. Let $f$ be the restriction of $g$ to $F$. By Lemmas~\ref{<4} and~\ref{<m+1}, we have that every $(k,q)$-colouring of $F$ is constricting. 

Let $\ell=|V(F)|-1$ and let $C=\{ v_0,v_1,\dots, v_\ell \}$ be the vertices of a cycle in $F$ such that 
\begin{itemize}
\item if $F$ is an odd cycle, then let $C=F$ (and the vertices are labelled in the natural order around $C$);
\item if $F$ is a clique, then choose the labelling such that $0\leq f(v_0)<f(v_1)<\dots <f(v_\ell)<k$.
\end{itemize} 
Now, let $g'$ be the colouring of $G$ defined by $g'(v)=k-g(v)\mod k$ for all $v\in V(G)$, and let $f'$ be the restriction of $g'$ to $F$. Then by Propositions~\ref{kdividessigma} and~\ref{sigma+sigma} we have that $k$ divides $\sigma(f)$ and $\sigma(f')$ and that $\sigma(f)+\sigma(f')=|V(F)|k$. In the case that $|V(F)|$ is odd, we get $\sigma(f)\neq\sigma(f')$ immediately. Otherwise, we have that $F$ is a clique and $|V(F)|\geq 4$. By our choice of $C$ we get $\sigma(f)=k$, and so $\sigma(f)\neq\sigma(f')$ in general. 

It follows that $g$ and $g'$ are in different components of $\mathscr{C}_{k,q}(G)$. If not, a path between $g$ and $g'$ in $\mathscr{C}_{k,q}(G)$ would indicate a path between $f$ and $f'$ in $\mathscr{C}_{k,q}(F)$. However, no such path can exist by Proposition~\ref{samesigma}.
\end{proof}

\section{Extending homomorphisms}
\label{ext}

As mentioned in the introduction, the motivation for studying circular mixing problems grew out of work on
extending circular colourings.  In the papers~\cite{Paint,AlbMoore,AlbWest}, the general question is as follows: 
Given $d,\ell\geq0$ and a $k$-colourable graph $X$ containing subgraphs $X_1, X_2, \dots, X_t$
together with a $(k+\ell)$-colouring of $X_1 \cup \cdots \cup X_t$,
can the $(k+\ell)$-colouring be extended to all of $X$, provided the distance
$d_{X}(X_i, X_j) \geq d$ for $i \neq j$?
For example, Kostochka proves a result in the affirmative when each $X_i$ is complete,
$\ell = 1$, and $d = 4k$ (as stated and proved in~\cite{AlbMoore}).  Albertson and West~\cite{AlbWest} study circular colourings where each $X_i$ is a single vertex, again
proving a positive result. They also conjecture such a theorem is possible when each $X_i$ is a circular clique.
In~\cite{BrewsterNoel}, we show that for a fixed $\ell$ no such theorem is possible. We now show how to use circular mixing to find such an $\ell$ which depends on $(k,q)$, provided $\gcd(k,q)=1$.

In~\cite{AlbWest,BrewsterNoel} the following product is a useful construction.
\begin{defn}
Let $G$ and $H$ be graphs.  The \emph{extension product} $G \bowtie F$
has as its vertex set $V(G) \times V(F)$ with $(g_1,f_1)(g_2,f_2)$ an edge
if $g_1 g_2 \in E(G)$ and either $f_1 f_2 \in E(F)$ or $f_1 = f_2$.   
\end{defn}
Alternatively one may view $G \bowtie F$ as the categorical product of
$G$ with a reflexive copy (a loop on each vertex) of $F$.  
Of particular importance for us is the product $G \bowtie P_n$, where $P_n$ 
is the path of length $n-1$ with vertex set $\{ 1, 2, \dots, n \}$.
It is straightforward to verify $G \bowtie F \to G$ via the projection onto
the first coordinate.  As a point of notation, for a fixed $i \in V(F)$,
the subgraph induced by $\{ (v, i) | v \in V(G) \}$ is isomorphic to $G$
and is denoted by $G_i$.  Given a homomorphism $\varphi: G \bowtie F \to H$,
the mapping defined by $\varphi_i(u) := \varphi(u,i)$ is a homomorphism $\varphi_i: G_i \to H$.

\begin{lem}\label{lem:homotopy}
Let $G$ be an $H$-colourable graph and $X = G \bowtie P_n$.  Suppose $\varphi_1: G_1 \to H$ and $\varphi_n: G_n \to H$
are homomorphisms.  Then there exists $\varphi:X \to H$ such that $\varphi|_{G_i} = \varphi_i$, $i \in \{ 1, n \}$,
i.e. there is an extension of $\varphi_1 \cup \varphi_n$ to all of $X$,
if and only if $d_{\mathscr{H}_H(G)}(\varphi_1, \varphi_n) < n$.
\end{lem}

\begin{proof}
Given two copies of $G$ and two homomorphisms $\psi_i: G_i \to H$, $i=1,2$, the single map 
$\psi_1 \cup \psi_2 : G_1 \cup G_2 \to H$
defines a homomorphism on $G \bowtie P_2$ if and only if $\psi_1 \psi_2$ is an edge of $\mathscr{H}_H(G)$.
The general result follows by induction.
\end{proof}

We remark the above result may be rephrased as $\varphi_1 \cup \varphi_n$ can be extended to
$G \bowtie P_n \to H$ if and only if $P_n \to \mathscr{H}_H(G)$ where the
end points of $P_n$ map to the precolourings $\varphi_1, \varphi_n$.

\begin{thm}
Suppose $G$ is $H$-colourable and $\mathscr{H}_H(G)$ is disconnected.  Then for any positive integer $d$, there exists a $G$-colourable
graph $X$ containing two subgraphs $X_1$ and $X_2$ such that $\varphi : X_1 \cup X_2 \to H$, $d_{X}(X_1, X_2) \geq d$,
but $\varphi$ does not extend to a homomorphism $X \to H$.
\end{thm}

\begin{proof}
Let $X = G \bowtie P_n$ where $n > d$.  Let $\varphi_i: G_i \to H$, $i \in \{ 1, n \}$, be precolourings with
$\varphi_1$ and $\varphi_n$ in different components of $\mathscr{H}_H(G)$.  Then by Lemma~\ref{lem:homotopy} there
is no extension of $\varphi_1 \cup \varphi_n$ to all of $X$.
\end{proof}

Following the spirit of the proof in~\cite{AlbWest}, we provide the general set-up for extending homomorphisms.  The graph
$H$ plays the role of $K_{k+\ell}$ and the graph $G$ plays the role of $K_k$.

\begin{lem}\label{lem:ext}
Let $X$ be a graph containing disjoint subgraphs $X_1,X_2,\dots,X_t$ (not necessarily connected, not necessarily isomorphic).
To each $X_i$ assign a precolouring $f_i:X_i\to H$, thus collectively define a precolouring 
$f_1 \cup \cdots \cup f_t : X_1 \cup \cdots \cup X_t \to H$.
Suppose there exists a graph $G$ and homomorphisms $\gamma:X\to G$ and $\varphi: G\to H$.  Let $\gamma_i = \gamma|_{X_i}$.
If
\begin{itemize}
\item for each $i$ there is a homomorphism $g_i: G\to H$ such that $f_i=g_i\gamma_i$; and
\item $d_{X}(X_i,X_j)\geq d_{\mathscr{H}_H(G)}(g_i,\varphi)+d_{\mathscr{H}_H(G)}(\varphi,g_j)$ for all $i\neq j$.,
\end{itemize}
then there exists an extension of $f_1 \cup \cdots \cup f_t$ to a homomorphism $f: X \to H$.
\end{lem}

\begin{proof}
Let $g_i : X_i \to G$ and let $g_i = g_i^{(0)}, g_i^{(1)}, \dots, g_i^{(n)} = \varphi$ be a shortest $g_i, \varphi$-path in $\mathscr{H}_H(G)$.
Given a vertex $v$ at distance $d \leq n$ from $X_i$, map $v$ to $H$ by $v \mapsto g_i^{(d)}\gamma(v)$.
It is easy to check that this defines a homomorphism of all vertices at distance at most $n$ from $X_i$ to $H$.  
Moreover, the vertices at distance $n$ are mapped to $H$ under $\varphi\gamma$.  Do this for each $i = 1, 2, \dots, t$.
For any unmapped vertices we can use $\varphi\gamma$.  This defines a homomorphism of $X \to H$ which extends 
$f_1 \cup \cdots \cup f_t$.
\end{proof}

A key requirement in the above lemma is that the precolouring $f_i: X_i \to H$ factors through $G$ as $f_i = g_i \gamma_i$.
A classic situation where we can be assured of this factoring is when each $X_i$ is isomorphic to the core of $X$.
Recall $\mathrm{core}(X)$ is the unique (up to isomorphism) minimal subgraph of $X$ to which $X$ admits a homomorphism, 
see~\cite{HellNesetril}.

\begin{thm}
\label{coreExt}
Let $X$ be an $H$-colourable graph containing disjoint subgraphs $X_1, X_2, \dots, X_t$ each isomorphic to $G := \mathrm{core}(X)$.
Further suppose $f_i : X_i \to H$ is a homomorphism for each $i=1,2, \dots, t$.  If $\mathscr{H}_H(G)$ is connected and
$d_{X}(X_i,X_j) \geq 2 \mathrm{rad}(\mathscr{H}_H(G))$, then the precolouring $f_1 \cup \cdots \cup f_t$ extends
to all of $X$.
\end{thm}

\begin{proof}
Let $G = \mathrm{core}(X)$ and $\gamma: X \to G$.  (Every graph admits a retraction to its core.)
Furthermore, since cores do not admit homomorphisms to a proper subgraph, each $\gamma_i: X_i \to G$ is an isomorphism.  
Thus we can let $g_i = f_i \gamma_i^{-1}$ and obtain $f_i = g_i \gamma_i$.
Let $\varphi$ be a centre in $\mathscr{H}_H(G)$.  Then $d_{\mathscr{H}_H(G)}(g_i,\varphi) + d_{\mathscr{H}_H(G)}(\varphi,g_j) 
\leq 2 \mathrm{rad}(\mathscr{H}_H(G)) \leq d_{X}(X_i,X_j)$.  By Lemma~\ref{lem:ext}, there is an extension of the
precolouring to an $H$-colouring of $X$.
\end{proof}

We briefly investigate the problem of extending precolourings of circular cliques. This problem was first suggested by a conjecture of Albertson and West~\cite{AlbWest}. In~\cite{BrewsterNoel}, the present authors disproved the conjecture and showed that the problem is more complicated than was previously anticipated. We make some progress in understanding these complications now. Let us begin with a lemma from~\cite{BrewsterNoel} rephrased
in the language of this paper.

\begin{lem}[Brewster and Noel~\cite{BrewsterNoel}]
\label{ceil}
If $k\geq 3(q-1)+1$, then $\mathfrak{M}(G_{k,q})\leq\left\lceil\frac{k}{q}\right\rceil+1$.
\end{lem}

The following proposition provides a general upper bound on $\mathfrak{M}_c(G_{k,q})$.

\begin{prop}
\label{mMax}
We have
\[\mathfrak{M}_c(G_{k,q})\leq\max\left\{\frac{k+1}{2}, \left\lceil\frac{k}{q}\right\rceil+1\right\}.\]
\end{prop}

\begin{proof}
By Theorem~\ref{v+1/2}, we have that 
\begin{equation}\label{v+1/2eqn}\mathfrak{M}_c(G_{k,q})\leq \max\left\{\frac{k+1}{2},\mathfrak{M}(G_{k,q})\right\}.\end{equation} 

In the case that $q\leq 2$, we have $k\geq 3(q-1)+1$ trivially and so $\mathfrak{M}(G_{k,q})\leq \left\lceil\frac{k}{q}\right\rceil+1$ 
by Lemma~\ref{ceil}. Thus, the result follows by Eq.~\eqref{v+1/2eqn} in this case. For $q\geq3$, the desired bound can be deduced 
from the following claim.

\begin{claim}
If $q\geq3$, then $\frac{k+1}{2}\geq\mathfrak{M}(G_{k,q})$.
\end{claim}

We divide the proof into two cases.

\begin{case}
Suppose that $k\geq 3(q-1)+1$.
\end{case}

By combining inequalities $q\geq3$ and $k\geq 3(q-1)+1$, we see that $k\geq7$. This implies that \[\frac{k+1}{2}\geq\frac{k+5}{3}\geq\left\lceil\frac{k}{3}\right\rceil+1\geq\left\lceil\frac{k}{q}\right\rceil+1\geq\mathfrak{M}(G_{k,q})\]
by Lemma~\ref{ceil}.

\begin{case}
Suppose that $k\leq 3(q-1)$.
\end{case}

In this case, we have $k\leq 3(q-1) < 4q-5$ which implies that
\[\frac{k+1}{2} > k-2q+3=\Delta(G_{k,q})+2=col(G_{k,q})+1\geq\mathfrak{M}(G_{k,q})\]
by Theorem~\ref{col}. This completes the proof of the claim and the proposition.
\end{proof}

The next theorem follows from Theorem~\ref{coreExt}, Proposition~\ref{mMax}, and the fact that 
$G_{k,q}$ is a core if $\gcd(k,q) = 1$.  (It is easy to see that the core of $G_{k,q}$ must be
a circular clique homomorphically equivalent to $G_{k,q}$.  By minimality of the core,
the assumption $\gcd(k,q)=1$ ensures that $G_{k,q}$ is itself a core.)

\begin{thm}\label{thm:extcirc}
For $k\geq2q$ and $\gcd(k,q)=1$, let $X$ be a $(k,q)$-colourable graph containing disjoint copies $X_1,X_2,\dots,X_t$ 
each isomorphic to $G_{k,q}$ and suppose that $\frac{k'}{q'}\geq \max\left\{\frac{k+1}{2}, \left\lceil\frac{k}{q}\right\rceil+1\right\}$. 
Then there exists a distance $d$ such that if $f_i:X_i\to G_{k',q'}$ is a homomorphism for $i=1,2,\dots,t$ and $d_{X}(X_i,X_j)\geq d$ 
for $i\neq j$, then the precolouring $f_1\cup\cdots\cup f_t$ extends to a $(k',q')$-colouring of $X$.
\end{thm}

We remark on the importance of the condition $\gcd(k,q)=1$ in the theorem to ensure $G_{k,q}$ is a core.  Consider the graph
$X := G_{6,2} + x$, where $x$ joins $\{ 0, 1, 4, 5 \}$ in $G_{6,2}$.  
The $4$-colouring $0, 1, 1 , 2, 2, 3$ of $G_{6,2}$ does not extend to a $4$-colouring
of the entire graph, despite $X$ being $(6,2)$-colourable and $G_{6,2}$ being $4$-mixing.

\section{Examples and discussion}
\label{Examples}

We show that there are graphs which either attain or nearly attain the bounds on $\mathfrak{M}_c$ and $\mathfrak{m}_c$ given by Theorems~\ref{col(G)Bound},~\ref{2Delta},~\ref{v+1/2} and~\ref{omega}. First, we observe that Theorems~\ref{v+1/2} and~\ref{omega} are sharp for cliques of size $3$ or greater.

\begin{prop}
We have
\[\mathfrak{m}_c(K_r)=\mathfrak{M}_c(K_r)=\left\{\begin{array}{ll} 	r & \mbox{ if }r\leq2, \\
																								r+1 & \mbox{ otherwise. }\end{array}\right.\]
\end{prop}

\begin{proof}
As mentioned in the proof of Proposition~\ref{K_mmtrees}, the case for $r\leq2$ is an easy exercise. For $r\geq3$ we have
\[r+1=\omega(K_r)+1\leq \mathfrak{m}_c(K_r)\leq \mathfrak{M}_c(K_r)\leq \max\left\{\frac{r+1}{2},\mathfrak{M}(K_r)\right\}=r+1\]
by 
Theorems~\ref{v+1/2} and~\ref{omega}, and the fact that $\mathfrak{M}(K_r)=r+1$. Thus, equality must hold throughout. 
\end{proof}

In a similar fashion, it can be seen that Theorems~\ref{2Delta} and~\ref{omega} are sharp for odd cycles; however, even cycles behave differently. 

\begin{prop}
We have
\[\mathfrak{M}_c(C_r) = \left\{		\begin{array}{ll} 	2 		& \text{if $r=4$},\\
														4 		& \text{otherwise},			\end{array}\right.\]
														and
\[\mathfrak{m}_c(C_r) = \left\{		\begin{array}{ll} 	2 		& \text{if $r$ is even},\\
														4 	& \text{otherwise}.			\end{array}\right.\]
\end{prop}

\begin{proof}
The fact that $\mathfrak{m}_c(C_4)=\mathfrak{M}_c(C_4)=2$ follows from Proposition~\ref{K_mmtrees} since $C_4\simeq K_{2,2}$. For $r\geq 5$, we have $\mathfrak{M}_c(C_r)\leq 4$ by Theorem~\ref{v+1/2}. If $r$ is odd, then we see that $\mathfrak{m}_c(C_r)=\mathfrak{M}_c(C_r)=4$ by Theorem~\ref{omega}. So, we are done if we can prove that, for even $r\geq 6$,
\begin{enumerate}[(a)]
\item $C_r$ is $(2q+1,q)$-mixing for all $q\geq \frac{r}{4}$, and\label{2q+1}
\item $C_r$ is not $(4q-1,q)$-mixing for any $q\geq1$.\label{4q-1} 
\end{enumerate}
In what follows, label the vertices of $C_r$ by $v_0,\dots, v_{r-1}$ in cyclic order and define $\tau$ and $\sigma$ as in Definition~\ref{tausigma}. 

Let us prove (\ref{2q+1}). Note that, for any $(2q+1,q)$-colouring $f$ of $C_r$, we must have $\tau(f,i)\in \{q,q+1\}$ for all $i$. Let $I_q(f):=\{i\in \{0,\dots,r-1\}: \tau(f,i)=q\}$ and $I_{q+1}(f):=\{0,\dots,r-1\}- I_{q}(f)$. We prove the following claim.

\begin{claim}
For every $(2q+1,q)$-colouring $f$ of $C_r$, we have $|I_q(f)|=|I_{q+1}(f)|=r/2$. 
\end{claim}

Clearly, 
\[\sigma(f) = |I_q(f)|q + |I_{q+1}(f)|(q+1) = rq + |I_{q+1}(f)|.\] 
Proposition 5.8 implies that $2q+1$ divides $\sigma(f)$. So, we let $a$ be an integer such that $a(2q+1) = rq + |I_{q+1}(f)|$. We obtain
\[|I_{q+1}(f)| = (2a-r)q +a.\]
If $a\geq r/2+1$, then $|I_{q+1}(f)| \geq 2q + (r/2+1)>r$ since $q\geq r/4$, which is a contradiction. On the other hand, if $a\leq r/2-1$, then  $|I_{q+1}(f)|\leq -2q + (r/2 - 1)<0$ which is, again, a contradiction. Therefore, $|I_{q+1}(f)| = a=r/2$, which proves the claim.

Therefore, every $(2q+1,q)$-colouring of $C_r$ is uniquely determined by specifying the value of $f(v_0)$ and the $r/2$ values of $i$ which satisfy $\tau(f,i)=q$. It is now not hard to see that all $(2q+1,q)$-colourings of $C_r$ can be generated by the mixing process.

Now, let us prove (\ref{4q-1}). By Lemma~\ref{<4} and Propositions~\ref{samesigma} and~\ref{sigma+sigma}, 
it suffices to exhibit a $(4q-1,q)$-colouring $f$ of $C_r$ such that $\sigma(f)\neq \frac{r(4q-1)}{2}$. We define
\[f(v_i) = \left\{		\begin{array}{ll} 	iq 		& \text{if $0\leq i\leq 3$},\\
										1 	& \text{if $i\geq4$ and $i$ is even},\\
										q+1 & \text{if $i\geq 4$ and $i$ is odd}.			\end{array}\right.\]
It is easily observed that $f$ is a $(4q-1,q)$-colouring. We calculate
\[\sigma(f) = 5q + \left(\frac{r-6}{2}\right)(4q-1) + (3q-2) = \left(\frac{r-2}{2}\right)(4q-1).\]
This completes the proof. 
\end{proof}

Next, we give a construction which shows that for any $d\geq2$ there are graphs of maximum degree $d$ which show that Theorem~\ref{2Delta} is sharp and that Theorem~\ref{v+1/2} is nearly sharp. In what follows, given $d,q\geq2$ we let \[k:=(2q-1)d+1\]
and define a graph $F_{d,q}$ on vertex set $\{0,1,\dots,k-1\}$ such that 
\[N(i)=\{i+q,i+q+(2q-1), \dots,i+q+(d-1)(2q-1)\}\]
for $0\leq i\leq k-1$. We remark that $F_{d,q}$ is a $d$-regular subgraph of $G_{k,q}$. The key feature of $F_{d,q}$ is that the identity map $\varphi$ on $\{0,1,\dots,k-1\}$ induces a frozen homomorphism $\varphi:F_{d,q}\to G_{k,q}$. In particular, this proves that $\mathfrak{M}_c(F_{d,q})\geq\frac{k}{q}$. 

\begin{prop}
\label{2DeltaBest}
For any $\varepsilon>0$ and $d\geq2$ there exists a $d$-regular graph $G$ with $\mathfrak{M}_c(G)\geq2d-\varepsilon$. 
\end{prop}

\begin{proof}
We have that $F_{d,q}$ is $d$-regular and
\[\mathfrak{M}_c(F_{d,q})\geq\frac{k}{q}=\frac{(2q-1)d+1}{q}\to 2d \text{ as } q\to\infty.\]
The result follows. 
\end{proof}

\begin{prop}
For any $d \geq 3$ there exists a $d$-regular graph $G$ with $\mathfrak{M}_c(G)\geq\frac{|V(G)|}{2}>\mathfrak{M}(G)$. 
\end{prop}

\begin{proof}
We have that $F_{d,2}$ is a $d$-regular graph on $k=3d+1$ vertices. In particular, 
\[\mathfrak{M}(F_{d,2})\leq d+2<\frac{3d+1}{2}=\frac{k}{2}\leq\mathfrak{M}_c(F_{d,2})\]
by Theorem~\ref{col} and the fact there is a frozen homomorphism mapping $F_{d,2}\to G_{k,2}$. The result follows.  
\end{proof}

An example of~\cite{Connectedness} shows that for every integer $m\geq4$ there is a graph $H_m$ which satisfies the following:
\begin{itemize}
\item $\chi(H_m) = \mathfrak{M}(H_m) = m$,
\item $|V(H_m)| = 2m-1$, and
\item $H_m$ contains a clique on $m-1$ vertices. 
\end{itemize}
By applying Theorems~\ref{v+1/2} and~\ref{omega} to their example, we obtain the following.

\begin{prop}
\label{H_m}
Let $m\geq2$ be an integer. There exists a graph $H_m$ such that 
\[\mathfrak{M}_c(H_m)= \chi(H_m) = m\]
if and only if $m\neq 3$. 
\end{prop}

\begin{proof}
For $m\geq4$, define $H_m$ as above. By Theorem~\ref{v+1/2}, we have
\[\mathfrak{M}_c(H_m)\leq \max\left\{\frac{|V(H_m)|+1}{2}, \mathfrak{M}(H_m)\right\} = m.\]
Also, since $H_m$ contains a clique on $m-1$ vertices, we have $\mathfrak{M}_c(H_m)\geq m$ by Theorem~\ref{omega}. Thus, the result holds for $m\geq4$. 

Now, by Theorem~\ref{omega} there can be no such graph when $m=3$. For the case that $m=2$ (ie. bipartite graphs), see Proposition~\ref{K_mmtrees}.
\end{proof}

\subsection{Questions for future study}

Given the known properties $\chi_c$ and $\chi_{c,\ell}$, one might wonder if analogous results hold for $\mathfrak{m}_c$ and $\mathfrak{M}_c$. In particular, one may ask questions of the following types:
\begin{enumerate}
\item Is the circular mixing threshold (number) always rational?
\item Under what conditions is the circular mixing threshold (number) an integer?
\item For which graphs is the circular mixing threshold (number) attained? 
\item What is the relationship between the mixing threshold (number) and circular mixing threshold (number)?
\end{enumerate}
Let us elaborate on the fourth question. We find the following problems to be of interest. Proposition~\ref{2DeltaBest} shows that $\mathfrak{M}_c(G)/\mathfrak{M}(G)$ can as large as $2-\varepsilon$ for any $\varepsilon>0$. It is not known, however, if this ratio can be arbitrarily large.

\begin{ques}
Does there exist a real number $r$ such that $\mathfrak{M}_c(G)\leq r\mathfrak{M}(G)$ for every graph $G$? If so, what is the smallest such $r$?
\end{ques}

Recall that every graph $G$ satisfies $\left\lceil\chi_c(G)\right\rceil =\chi(G)$. However, the analogous statement for the mixing number is false; as we have seen, $\mathfrak{m}_c = \mathfrak{m}-1$ for trees on at least two vertices and complete bipartite graphs and $\mathfrak{m}_c = \mathfrak{m}-2$ for even cycles of length at least six. However, all of these examples are bipartite; we wonder about the relationship between the mixing number and circular mixing number for non-bipartite graphs.

\begin{ques}
Is it true that $\left\lceil \mathfrak{m}_c(G)\right\rceil = \mathfrak{m}(G)$ for every non-bipartite graph $G$?
\end{ques}

Currently, we do not know of any bipartite graphs for which $\mathfrak{m}_c > 2$. We conjecture the following.

\begin{conj}
For every bipartite graph $G$ there exists $q_0$ such that for every $q\geq q_0$, $G$ is $(2q+1,q)$-mixing. 
\end{conj}

Graphs which dismantle to $K_2$ are the only examples of graphs that we have satisfying $\mathfrak{M}_c<\mathfrak{M}$. This suggests the following question.

\begin{ques}
Does there exist a graph $G$ which does not dismantle to $K_2$ such that $\mathfrak{M}_c(G)<\mathfrak{M}(G)$?
\end{ques}

We also wonder if the circular mixing threshold can be as small as the circular chromatic number (cf. Proposition~\ref{H_m}) for non-trivial graphs which do not dismantle to $K_2$.

\begin{ques}
Does there exist a graph $G$ which contains at least one edge and does not dismantle to $K_2$ such that $\mathfrak{M}_c(G)=\chi_c(G)$?
\end{ques}

\subsubsection*{Acknowledgement}

The authors wish to thank the anonymous referees for many helpful suggestions to improve the paper.

\bibliographystyle{plain}
\bibliography{MixingPaperNoExt}

\end{document}